\documentclass[10pt]{article}
\usepackage[english]{babel}
\usepackage[toc,page]{appendix}
\usepackage{bbm}
\usepackage{fullpage}
\usepackage[top=1in, bottom=1.25in, left=1in, right=1in]{geometry}
\usepackage{etaremune}
\usepackage{enumitem}
\usepackage{graphicx}
\usepackage[labelformat=empty]{caption}
\usepackage{subcaption}
\usepackage{tikz}

\usepackage{framed}

\usepackage{amssymb,amsmath,amsthm}

\usepackage{hyperref}

\usetikzlibrary{decorations.pathreplacing,calligraphy}

\DeclareSymbolFont{rsfs}{U}{rsfs}{m}{n}
\DeclareSymbolFontAlphabet{\mathscrsfs}{rsfs}
\usepackage{mathrsfs}

\usepackage{url}

\hypersetup{
  colorlinks   = true, 
  urlcolor     = blue, 
  linkcolor    = blue, 
  citecolor   = red 
}

\newtheorem{theorem}{Theorem}[section]
\newtheorem{lemma}[theorem]{Lemma}

\newtheorem{proposition}[theorem]{Proposition}
\newtheorem{corollary}[theorem]{Corollary}

\theoremstyle{definition}
\newtheorem{definition}{Definition}

\newtheorem{remark}[theorem]{Remark}

\numberwithin{equation}{section}

\usepackage{times}

\newcommand{\bea}{\begin{eqnarray}}
\newcommand{\eea}{\end{eqnarray}}
\newcommand{\<}{\langle}
\renewcommand{\>}{\rangle}

\newcommand{\wt}{\widetilde}

\newcommand{\wh}{\widehat}

\def\ie{\text{i.e.~}}

\def\iid{\text{i.i.d.~}}

\def\eps{{\varepsilon}}

\def\con{{\small\mathsf{con}}}
\def\ord{{\small\mathsf{ord}}}
\def\stoc{{\small\mathsf{stoc}}}
\def\PSD{{\small\mathsf{PSD}}}

\def\cG{{\mathcal G}}

\def\cC{{\mathcal C}}

\def\cO{{\mathcal O}}

\def\de{{\rm d}}

\def\<{\langle}
\def\>{\rangle}

\def\diam{{\rm diam}}

\def\b0{{\boldsymbol{0}}}

\def\cI{{\mathcal I}}
\def\cS{{\mathcal S}}

\def\cB{{\mathcal B}}

\renewcommand{\b}{\mathbf{b}}

\def\lrarrow{\leftrightarrow}

\def\lt{\left}
\def\rt{\right}

\def\la{\langle}
\def\ra{\rangle}

\def\eps{\varepsilon}

\def\bbE{{\mathbb{E}}}

\def\bbN{{\mathbb{N}}}
\def\bbP{{\mathbb{P}}}
\def\bbR{{\mathbb{R}}}
\def\bbS{{\mathbb{S}}}

\def\bbZ{{\mathbb{Z}}}

\def\cB{{\mathcal{B}}}

\def\cP{{\mathcal{P}}}

\usepackage{mathtools} 
\usepackage{dsfont} 
\usepackage{stmaryrd} 

\author{
    Mark Sellke
}

\title{Localization of Random Surfaces with Monotone Potentials
\\
and an FKG-Gaussian Correlation Inequality}

\date{}

\begin{document}

\maketitle

\begin{abstract}
    \noindent
    The seminal 1975 work of Brascamp, Lieb and Lebowitz 
    initiated the rigorous study of Ginzberg--Landau random surface models.  
    It was conjectured therein that fluctuations are \emph{localized} on $\bbZ^d$ when $d\geq 3$ for very general potentials, matching the behavior of the Gaussian free field.
    We confirm this behavior for all even potentials $U:\bbR\to\bbR$ satisfying $U'(x)\geq \min\big(\eps x,\frac{1+\eps}{x}\big)$ on $x\in \bbR^+$.
    Given correspondingly stronger growth conditions on $U$, we show power or stretched exponential tail bounds on all transient graphs, which determine the maximum field value up to constants in many cases.
    Further extensions include non-wired boundary conditions and iterated Laplacian analogs such as the membrane model.
    Our main tool is an FKG-based generalization of the Gaussian correlation inequality, which is used to dominate the finite-volume Gibbs measures by mixtures of centered Gaussian fields.
\end{abstract}

\section{Model and Main Results}

We study the discrete Ginzberg--Landau or $\nabla\phi$ model on a locally finite, connected graph $G$.
For each edge $e=\{v,v'\}\in E(G)$, let $U_e:\bbR\to\bbR$ be a potential function with $U_e(x)=U_e(-x)$ for all $x\in \bbR$.
Given a finite subset $\Lambda\subset V(G)$, let $G_{\Lambda}$ be the multi-graph obtained by contracting $V(G)\backslash \Lambda$ to a single vertex $z_{\Lambda}$.
Our main objects of study will be the wired Gibbs measures on fields $\phi:G_{\Lambda}\to\bbR$, where we pin $\phi(z_{\Lambda})=0$:
\begin{equation}
\label{eq:model-def}
    \de\mu_{G_{\Lambda}, \vec U}(\phi)
    \equiv
    \frac{1}{Z_{G_{\Lambda}, \vec U}} 
    \exp \bigg(
    -\sum_{e=\{v,v'\}\in E(G_{\Lambda})} U_e\big(\phi(v)-\phi(v')\big)
    \bigg) 
    \delta_0(\phi(z_{\Lambda}))
    \prod_{v \in \Lambda} \de \phi(v)
    .
\end{equation}
Here and throughout we write $\vec U$ for $(U_e)_{e\in E(G)}$.

When $U_e(x)=cx^2$ for all $e\in E(G)$, \eqref{eq:model-def} is the discrete Gaussian free field which is well understood. 
It is natural to expect some degree of universality: the large scale behavior should depend primarily on the graph $G$ rather than the potential $U$.
We focus on one of the most basic such questions, namely the order of fluctuations of $\phi(v)$.

This problem was put forward in \cite{brascamp1975statistical} (in the case that the potential $U$ is the same for all edges), where it was shown by a Mermin--Wagner argument that when $\Lambda=[-L,\dots,L]^2\subseteq\bbZ^2$, the fluctuations of $\phi(\vec 0)$ have order $\Omega(\sqrt{\log L})$ for extremely general $U$; see also \cite{frohlich1981absence,ioffe20022d,milos2015delocalization} for extensions.
For $d\geq 3$, under the much stronger assumption that $U$ is \emph{uniformly} convex on $\bbR$ (with uniformly positive second derivative), \cite{brascamp1975statistical} showed the fluctuations of $\phi(v)$ are \emph{localized} at scale $O(1)$ independently of $\Lambda$.
However they conjectured that localization should hold much more broadly, perhaps for all even $U:\bbR\to\bbR$ such that $\int_{\bbR} e^{-\alpha U(x)}\de x<\infty$ for all $\alpha>0$.
This conjecture has remained a significant outstanding challenge, and was also highlighted in a more recent survey as \cite[Open Problem 1]{velenik2006localization}. 
As reviewed in Subsection~\ref{subsec:related} below, even the case of (non-uniformly) convex potentials has not been fully resolved despite much effort. For instance a notable recent result of \cite{magazinov2022concentration} proved localization on the torus for convex $U$ such that $U''(x)>0$ for almost every $x$.

Our first main result, Theorem~\ref{thm:main} below, establishes localization for a large class of potentials. 
All that is required of $U$ is a \emph{monotonicity} condition, which is qualitatively more permissive than convexity.
The $1+\eps$ factor below is natural as it ensures integrability of $e^{-U(x)}$ so that the Gibbs measure \eqref{eq:model-def} is defined for all $G$. 
(Indeed this does not even imply the integrability condition mentioned above, which corresponds to super-logarithmic growth.)

\begin{definition}
\label{def:eps-monotone}
    The increasing function $U:[0,\infty)\to [0,\infty)$ is $\eps$-monotone if at all points of differentiability,
    \[
    U'(x)\geq 
    \min\Big(\eps x,\frac{1+\eps}{x}\Big).
    \]
    The even function $U:\bbR\to [0,\infty)$ is $\eps$-monotone if $U|_{[0,\infty)}$ is.
    $\vec U=(U_e)_{e\in E(G)}$ is $\eps$-monotone if each $U_e$ is.
\end{definition}

\begin{definition}
\label{def:percolation-graphs}
    The connected locally finite graph $G$ is percolation transient if for some $p=p(G)<1$, $p$-Bernoulli bond percolation on $G$ almost surely has at least one transient infinite cluster. 
\end{definition}

Percolation transience of $\bbZ^d$ for $d\geq 3$ was famously shown in \cite{grimmett1993random}.
Recently, Hutchcroft \cite{hutchcroft2023transience} has shown this property for all \emph{transitive} transient graphs.
Our first main result shows localization for $\eps$-monotone $\vec U$ on such graphs.

\begin{theorem}
\label{thm:main}
    Fix $\eps>0$ and let $G$ be a percolation transient graph with $v\in V(G)$. 
    Then as the $\eps$-monotone $\vec U$ and vertex subset $\Lambda\subseteq V(G)$ vary, the set of $\mu_{G_{\Lambda},\vec U}$-laws of $\phi(v)$ is tight.
\end{theorem}

Such tightness directly implies existence of an infinite volume Gibbs measure on $G$, which is also a natural definition of localization. Namely it suffices to take any weak subsequential limit of the $\mu_{G_{\Lambda},\vec U}$ as $\Lambda\uparrow V(G)$.

On $\bbZ^2$, although there is delocalization, we obtain upper bounds of the same order as for the Gaussian free field.

\begin{theorem}
\label{thm:Z2}  
    Fix $\eps>0$ and let $G=\bbZ^2$.
     Then as the $\eps$-monotone $\vec U$ and vertex subset $\Lambda\subseteq V(G)$ vary, the set of $\mu_{G_{\Lambda},\vec U}$-laws of $\frac{\phi(0,0)}{\sqrt{\log \diam(\Lambda)}}$ is tight.
\end{theorem}

Before stating further extensions, let us briefly describe the main proof strategy.
The first step is to write any $\eps$-monotone $U$ as $U=V+W$ where $V,W:\bbR\to\bbR$ are even and increasing on $\bbR_+$. 
Proposition~\ref{prop:monotone-function-representation} shows this can be done so that additionally $e^{-V}$ is a mixture of centered Gaussian densities, i.e.
\begin{equation}
\label{eq:V-GMM}
    e^{-V(x)}
    =
    \int_{0}^{\infty}
    \frac{e^{-x^2/2\kappa^2}}{\kappa\sqrt{2\pi}}
    \de \rho(\kappa)
\end{equation}
for an absolutely continuous probability measure $\rho$ on $(0,\infty)$ (with $\int \frac{\de\rho(\kappa)}{\kappa}<\infty$).
The $\eps$-monotonicity of $\vec U$ enters only to show such a $\rho$ exists (and does not depend on the edge $e$).

From the representation \eqref{eq:V-GMM}, it follows that $\mu_{G_{\Lambda},V}$ decomposes into a mixture of weighted Gaussian free fields, each described by a choice of edge resistances on $E(G_{\Lambda})$.
This observation was previously made and exploited to study Gibbs measures of the form $\mu_{G_{\Lambda},V}$ in \cite{biskup2007phase,biskup2011scaling} and subsequent works (see Subsection~\ref{subsec:related}).
However the joint distribution of these resistances may be complicated, and it is not clear whether this description of $\mu_{G_{\Lambda},V}$ helps to understand $\mu_{G_{\Lambda},\vec U}$.
Using the \emph{FKG-Gaussian correlation inequality} developed in Subsection~\ref{subsec:FKG-GCI}, we reduce localization of $\mu_{G_{\Lambda},\vec U}$ to localization of $\mu_{G_{\Lambda},V}$. Moreover we are able to replace the complicated joint distribution of edge resistances for $\mu_{G_{\Lambda},V}$ by \iid resistances.
The resulting \iid model is simple enough to analyze directly, e.g. the edges with bounded resistance form a Bernoulli bond percolation.
The extensions in the next subsection come from different analyses of Gaussian free fields with \iid resistances, following the same type of initial reduction.

\subsection{Extensions}

Here we state various extensions of Theorem~\ref{thm:main}.  
Under stronger growth conditions on $U$, which are used to obtain tail estimates for $\rho$ in \eqref{eq:V-GMM}, we show power law and stretched exponential upper tail bounds for $\phi(v)$.
These are valid on all transient graphs: thus localization of the Gaussian free field always implies localization for a wide class of monotone potentials, which can diverge as slowly as $U(x)\approx (3+\eps)\log(x)$. 
In many cases, these bounds allow us to determine the typical order of $\max_{v\in\Lambda}|\phi(v)|$ as $|\Lambda|\uparrow\infty$.
Finally we give extensions to non-wired boundary conditions and iterated Laplacian random surfaces such as the membrane model.

\paragraph{Tail Bounds for Field Values on Transient Graphs.}

Theorem~\ref{thm:main} does not give concrete tail bounds for $\phi(v)$.
In fact if $U$ diverges slowly, $\phi(v)$ can have a very heavy tail.
For example suppose that $U(x)=(\alpha+1)\log x \pm O(1)$ for large $x$ and $\alpha>0$. With $N(v)\subseteq V(G)$ the neighborhood of $v$, conditioning on the restriction $\phi|_{N(v)}$ easily shows
\begin{equation}
\label{eq:fluctuation-LB} 
\bbP\Big[\max\limits_{w\in v\cup N(v)}|\phi(w)|\geq t\Big]
\geq 
(t+1)^{-|N(v)|\alpha}/C
\end{equation}
for any Gibbs measure, with $C=C(U,\alpha,|N(v)|)$. (See the proof of Corollary~\ref{cor:maximum-value-Zd} for details.)

Next we show an upper bound for the tails of $\phi(v)$ under such a growth condition. 
We use $R_G(v\lrarrow w)$ to denote the effective resistance in the unweighted graph $G$ (see Subsection~\ref{subsec:setup} for reminders, or \cite{lyons2017probability}).

\begin{definition}
\label{def:power-tail}
A random variable $\xi$ has \textbf{sub-$\alpha$ tail} if $\bbP[|\xi|\geq t]\leq Ct^{-\alpha}$ for some $C<\infty$ and all $t$.
A family of random variables has uniformly sub-$\alpha$ tails if the constant $C$ is uniform.
\end{definition}

\begin{definition}
\label{def:k-eps-monotone}
The increasing function $U:[0,\infty)\to [0,\infty)$ is $(\alpha,\eps)$-monotone if $U'(x)\geq \min\lt(\eps x,\frac{\alpha+1}{x}\rt)$ at all points of differentiability $x>0$.
$\vec U=(U_e)_{e\in E(G)}$ is $(\alpha,\eps)$-monotone if each $U_e$ is.
\end{definition}

\begin{theorem}
\label{thm:moment-bounds}
    Suppose $\vec U$ is $(\alpha,\eps)$-monotone for $\alpha>2$ and $\eps>0$.
    Then for $v\in \Lambda\subseteq V(G)$, the $\mu_{G_{\Lambda},\vec U}$-law of $\phi(v)$ has sub-$\alpha$ tail with constant $C\big(\alpha,\eps,R_{G_{\Lambda}}(v\lrarrow z_{\Lambda})\big)$.
    In particular if $G$ is transient, then $\phi(v)$ has uniformly sub-$\alpha$ tail as $\Lambda\subseteq V(G)$ varies.
\end{theorem}

While \eqref{eq:fluctuation-LB} and Theorem~\ref{thm:moment-bounds} are both power-law bounds, their exponents differ unless $v$ has degree $|N(v)|=1$ (but see Theorem~\ref{thm:moment-bounds-Zd} below).
If $U$ grows at a sub-quadratic polynomial rate, we obtain stretched exponential decay.

\begin{definition}
\label{def:stretched-tail}
A non-negative random variable $\xi$ has \textbf{stretched sub-$\beta$ tail} if $\bbP[\xi\geq t]\leq Ce^{-t^{\beta}/C}$ for some $C<\infty$ and all $t$.
A family of random variables has uniformly stretched sub-$\beta$ tails if this holds with a uniform constant $C$.
\end{definition}

\begin{definition}
    The increasing $U:[0,\infty)\to [0,\infty)$ is $(\beta,\eps)$-polynomially monotone if $U'(x)\geq \eps\,\min(x,x^{\beta-1})$ at all points of differentiability $x>0$.
    $\vec U=(U_e)_{e\in E(G)}$ is $(\beta,\eps)$-polynomially monotone if each $U_e$ is.
\end{definition}

\begin{theorem}
\label{thm:stretched-exponential-bounds}
    Fix $\beta\in (0,2]$ and $\eps>0$, and $(\beta,\eps)$-polynomially monotone $\vec U$ on the transient graph $G$.
    Then for all $v\in V(G)$, the $\mu_{G_{\Lambda},\vec U}$-law of $\phi(v)$ has stretched sub-$\beta$ tail with constant $C=C(\beta,\eps,R_G(v\lrarrow \infty))$.
\end{theorem}

The $\beta=2$ case is essentially immediate from the Gaussian correlation inequality (as explained in the proof), and already encompasses uniformly convex $U$.
Furthermore the $\beta=1$ case includes all convex functions $U$ with $U(x)/x^2$ bounded below near zero.
For such $U$, the associated fields $\phi(v)$ are therefore sub-exponential (which also follows from weaker localization statements thanks to log-concavity). 
Interestingly in \cite{magazinov2022concentration}, the tails of $\phi(v)$ are shown to decay strictly faster than Gaussian when $U(x)=x^2+|x|^p$ with $p>2$. Since our techniques are based on Gaussian domination, they cannot recover such behavior without additional work.

\paragraph{Order of the Maximum Value on $\Lambda$.}

Theorem~\ref{thm:stretched-exponential-bounds} is tight, since in the stretched exponential analog of \eqref{eq:fluctuation-LB}, $N(v)$ enters more mildly.
This allows us to determine the order of $\max_{v\in\Lambda}|\phi(v)|$ for $U$ of suitably regular growth.

\begin{corollary}
\label{cor:maximum-value}
    Let $G$ be transient and transitive.
    Suppose $\vec U$ is $(\beta,\eps)$-polynomially monotone for $\beta\in (0,2]$ and $\eps>0$, with $0\leq U_e(x)\leq C(x^{\beta}+1)$ for all $(x,e)\in\bbR_+\times E(G)$.
    Then the $\mu_{G_{\Lambda_i},\vec U}$-laws of
    \[
    (\log |\Lambda_i|)^{-1/\beta}\max_{v\in \Lambda_i}|\phi(v)|
    \]
    are tight for any sequence $\Lambda_i\subseteq G$ with $|\Lambda_i|\uparrow\infty$, and all subsequential limits have compact support in $(0,\infty)$.
\end{corollary}

For power law tails, we can similarly identify the maximum value under an extra condition we call strong transience. 
Say that $G$ is $k$-transient at $v$ if $G$ contains $k$ connected, edge-disjoint transient subgraphs, each containing $v$. 
If $G$ is transitive of degree $D$ and $D$-transient (at arbitrary $v$), we say $G$ is \textbf{strongly transient}.

\begin{theorem}
\label{thm:moment-bounds-Zd}
    Let $G$ be $k$-transient at $v$.
    Suppose $\vec U$ is $(\alpha,\eps)$-monotone for $\alpha>2$ and $\eps>0$.
    Then for $v\in \Lambda\subseteq \bbZ^d$ with $d\geq 3$, the $\mu_{G_{\Lambda},\vec U}$-law of $\phi(v)$ has sub-$k\alpha$ tail.
    Moreover $\bbZ^d$ and $d$-regular trees are strongly transient for $d\geq 3$.
\end{theorem}

\begin{corollary}
\label{cor:maximum-value-Zd}
    Let $G$ be a degree $D$ strongly transient graph.
    Suppose $\vec U$ is $(\alpha,\eps)$-monotone for $\alpha>2$ and $\eps>0$, with $0\leq U_e(x)\leq \alpha\log(x+1)+C$ for all $(x,e)\in\bbR_+\times E(G)$.
    Then the $\mu_{G_{\Lambda_i},\vec U}$-laws of
    \[
    |\Lambda_i|^{-\frac{1}{D\alpha}}\max_{v\in \Lambda_i}|\phi(v)|
    \]
    are tight in $(0,\infty)$ for any sequence of subsets $\Lambda_i\subseteq V(G)$ with $|\Lambda_i|\uparrow\infty$.
\end{corollary}

The above bounds are two-sided (since $0\notin (0,\infty)$), but we do not know how to remove the absolute value around $\phi(v)$.
In other words, large field fluctuations might occur only in one (random) direction.

\paragraph{Generalized Zero Boundary Conditions.}

So far we have dealt only with wired boundary conditions.
Here we consider more general cases including free boundary conditions, as well as the torus when $G=\bbZ^d$.\footnote{In the former case one takes $S=\{\phi~:~\phi(v_0)=0\}$ in Definition~\ref{def:generalized-boundary}; in the latter one takes $\Lambda=[-L,\dots,L]^d$ and uses $S$ to identify boundary vertices. One could also, for example, pin a larger subset of field values to zero or identify more vertices together.}
However, stemming from our use of the Gaussian correlation inequality, we can handle only centered boundary conditions.

\begin{definition}
\label{def:generalized-boundary}
Fix a finite connected subgraph $\Lambda\subseteq V(G)$, linear subspace $S\subseteq\bbR^{\Lambda}$, and $v_0\in V(G)$ with $\phi(v_0)=0$ for all $\phi\in S$.
With $\de\lambda_S$ Lebesgue measure on $S$, the \textbf{generalized zero boundary Gibbs measure} $\wh\mu_{\Lambda,S,v_0,U}$ on $S$ is 
\[
\wh\mu_{\Lambda,S,v_0,U}
=
\frac{1}{\wh Z_{\Lambda,S,v_0,U}}
\exp \bigg(
    -\sum_{e=\{v,v'\}\in E(\Lambda)} U_e\big(\phi(v)-\phi(v')\big)
    \bigg) 
    \de\lambda_S(\phi).
\]
\end{definition}

For such boundary conditions, we obtain essentially the same results as before, with slightly stronger conditions on $G$ in the case of Theorem~\ref{thm:main}.
However tightness is shown below only after taking the limit $\Lambda\uparrow V(G)$.
This difference is to be expected: for Gaussian free fields, Rayleigh's monotonicity principle states that a smaller domain has better localization under wired boundary conditions, but worse localization under free boundary conditions.

\begin{theorem}
\label{thm:free-boundary}
    For any $\eps>0$, the following holds for any $\eps$-monotone $\vec U$.
    Suppose $G$ is transitive and percolation transient with almost surely unique infinite percolation cluster for some $p<1$, and fix $v_0\in V(G)$.
    Let $\big(\wh\mu_{\Lambda_i,S_i,v_0,U}\big)_{i\in\bbN}$ be a sequence of generalized zero boundary Gibbs measures with $\Lambda_i\uparrow V(G)$, and $v_0\in S_i$ for all $i$.
    \begin{enumerate}
    \item The set $\cG$ of infinite volume Gibbs measures that are weak subsequential limits $\lim\limits_{k\to\infty}\wh\mu_{\Lambda_{i_k},S_{i_k},v_0,U}$ is non-empty.
    \item As $\nu\in\cG$ and $v\in V(G)$ vary, the $\nu$-laws of $\phi(v)$ form a tight family in $\cP(\bbR)$.
    \end{enumerate}
\end{theorem}

\begin{theorem}
\label{thm:free-boundary-tail}
    Let $G$ be a transient graph and fix $v,v_0\in V(G)$.
    Let $\big(\wh\mu_{\Lambda_i,S_i,v_0,U}\big)_{i\in\bbN}$ be a sequence of generalized zero boundary Gibbs measures with $\Lambda_i\uparrow V(G)$, and $v_0\in S_i$ for all $i$.
    In all cases below, the set $\cG$ of infinite volume Gibbs measures that are weak subsequential limits $\lim\limits_{k\to\infty}\wh\mu_{\Lambda_{i_k},S_{i_k},v_0,U}$ is non-empty. Moreover:
    \begin{enumerate}[label = {(\Alph*)}]
    \item 
    \label{it:alpha-eps-free}
    If $U$ is $(\alpha,\eps)$-monotone for $\alpha> 2$, then the $\nu$-laws of $\phi(v)$ have sub-$\alpha$ tail, uniformly in $\nu\in\cG$.
    \item 
    \label{it:beta-eps-free}
    If $U$ is $(\beta,\eps)$-polynomially monotone for $\beta\in (0,2]$, then then the $\nu$-laws of $\phi(v)$ have stretched sub-$\beta$ tail, uniformly in $\nu\in\cG$.
\end{enumerate}
    In both cases the constant implicit in the uniform tail decay of $\phi(v)$ depends only on $\alpha$ (or $\beta$), $\eps$, and $R_G(v\lrarrow v_0)$. In particular they are uniform in $v$ if $G$ is transitive.
\end{theorem}

\paragraph{Iterated Laplacian Random Surfaces.}

Since Theorem~\ref{thm:main} involves percolation, the underlying graph is genuinely relevant. 
On the other hand Theorems~\ref{thm:moment-bounds} and \ref{thm:stretched-exponential-bounds}, on $(\alpha,\eps)$-monotone and $(\beta,\eps)$-polynomially monotone $\vec U$, have little to do with the graph structure. For Hamiltonians which are sums of functions $U_e$ applied to linear functionals $\ell_e(\phi)$ of the field, they only require that the corresponding Gaussian model localizes.
A precise general statement to this effect is in Theorem~\ref{thm:convex-comparison} later. 
We present here its consequences for the iterated Laplacian models, as introduced in the Gaussian case by \cite{sakagawa2003entropic} (see also \cite{lodhia2016fractional}).
In particular the bi-Laplacian ``membrane'' model ($j=2$ below) has received significant recent attention \cite{kurt2009maximum,bolthausen2017exponential,cipriani2019scaling,schweiger2020maximum,cipriani2021scaling}. 
Some standard techniques for the $\nabla\phi$ model are no longer available here, as discussed in \cite[Section 1.1]{thoma2021thermodynamic}.

Define as usual the discrete Laplacian $\Delta \phi(x)=\sum_{y\sim x}[\phi(x)-\phi(y)]$.
Then $\Delta^{j/2}\phi$ is defined iteratively for $j\in 2\bbN$, while for odd $j\in 2\bbN+1$ we set $\Delta^{j/2}\phi=\nabla\big(\Delta^{(j-1)/2}\phi\big)$.
If $j$ is odd (as in \eqref{eq:model-def}), given $\vec U=\{U_e\}_{e\in E(G)}$ let
\begin{equation}
\label{eq:model-def-j-odd}
    \de\mu_{G,\Lambda \vec U,j}(\phi)
    \equiv
    \frac{1}{Z_{G,\Lambda \vec U,j}(\phi)} 
    \exp \bigg(
    -\sum_{e\in E(G)} U_e\big(\Delta^{j/2}\phi(e)\big)
    \bigg) 
    \prod_{v \in \Lambda} \de \phi(v)
    \prod_{v'\in V(G)\backslash\Lambda}
    \delta_0(\phi(v')
    .
\end{equation}
If $j$ is even, given $\vec U=\{U_w\}_{w\in V(G)}$ let 
\begin{equation}
\label{eq:model-def-j-even}
    \de\mu_{G_{\Lambda}, \vec U,j}(\phi)
    \equiv
    \frac{1}{Z_{G,\Lambda \vec U,j}(\phi)} 
    \exp \bigg(
    -\sum_{w\in V(G)} U_w\big(\Delta^{j/2}\phi(w)\big)
    \bigg) 
    \prod_{v \in V} \de \phi(v)
    \prod_{v'\in V(G)\backslash\Lambda}
    \delta_0(\phi(v')
    .
\end{equation}
Note that since iterated Laplacians depend on larger $G$-neighborhoods, we have not contracted $G$ to $G_{\Lambda}$. Hence the sums above have infinitely many terms, but only finitely many are non-zero so $\mu_{G_{\Lambda}, \vec U,j}$ is defined.

\begin{theorem}
\label{thm:higher-laplacian}
    Let $G$ be a transitive graph with finite Green's function for $\Delta^j$. Then:
    \begin{enumerate}[label = {(\alph*)}]
        \item 
        \label{it:membrane-moment}
        If $U$ is $(\alpha,\eps)$-monotone for $\alpha> 2$, then the $\mu_{G,\Lambda,\vec U,j}$-tails of $\phi(v)$ are uniformly sub-$\alpha$ as $\Lambda\subseteq V(G)$ varies.
        \item 
        \label{it:membrane-stretch}
        If $U$ is $(\beta,\eps)$-polynomially monotone for $\beta\in (0,2]$, then the $\mu_{G,\Lambda,\vec U,j}$-tails of $\phi(v)$ are uniformly stretched sub-$\beta$ as $\Lambda\subseteq V(G)$ varies.
        \item 
        \label{it:membrane-maximum}
        Corollary~\ref{cor:maximum-value} continues to hold for $\mu_{G,\Lambda,\vec U,j}$ in place of $\mu_{G_{\Lambda},\vec U}$ (for the same class of $\vec U$).
    \end{enumerate}
\end{theorem}

Theorem~\ref{thm:higher-laplacian} applies for $\bbZ^d$ when $d\geq 2j+1$ as shown by \cite{sakagawa2003entropic}. Meanwhile the Gaussian model delocalizes for $d\leq 2j$.
\cite{cipriani2023maximum} recently studied the Gaussian membrane model ($j=2$) on $m$-regular trees for $m\geq 3$, showing finiteness of the Green's function there as well.
As in \cite{sakagawa2003entropic}, one may also consider Hamiltonians $H(\phi)=\sum_{j=1}^J H_j(\phi)$ that combine multiple values of $j$.
When all potentials are monotone, the localization results of Theorem~\ref{thm:higher-laplacian} hold for such $H$ if there exists an individual $H_j$ satisfying the relevant hypothesis.

\subsection{Related Work}
\label{subsec:related}

The main localization result of \cite{brascamp1975statistical} requires $U$ to be uniformly convex, and even the extension to general convex $U$ remains open despite significant effort.
\cite{bricmont1982surface} proved localization for $U(x)=|x|$ using infrared bounds.
\cite{magazinov2022concentration} handled convex functions with almost-everywhere nonzero second derivative on the torus, in particular resolving the cases $|x|^p$ with $p>1$. \cite{dario2023upper} proved localization for a class of convex potentials with asymptotically super-quadratic growth that can be linear on large finite intervals.
Several other results for mild perturbations of convex potentials (not all on localization) are available in \cite{brydges1990grad,cotar2009strict,cotar2012decay,adams2016strict}.

The behavior of $\max_{v\in\Lambda}\phi(v)$ has been intensely studied on $\bbZ^2$ for the Gaussian free field as surveyed in \cite{biskup2020extrema}.
For uniformly convex $\nabla\phi$ models this was studied in \cite{belius2016maximum}.

Other aspects of $\nabla\phi$ models are of interest besides localization. For uniformly convex potentials, \cite{funaki1997motion} gave a mean-curvature description of Langevin dynamics, proving along the way that ergodic Gibbs measures are uniquely determined by their slopes; see \cite{yau1991relative,funaki2001large,armstrong2022quantitative} for further results on dynamics. 
A large deviations principle was obtained in \cite{deuschel2000large}, while
\cite{sheffield2005random} unified and extended these results to encompass $\bbZ$-valued fields.
Recently \cite{armstrong2022c2,armstrong2022quantitative} proved the surface tension (free energy) is a $C^2$ function of the slope for uniformly convex potentials.
Another important phenomenon is macroscopic convergence of the fluctuations to the continuum Gaussian free field, shown under various conditions in \cite{naddaf1997homogenization,giacomin2001equilibrium,hilger2016scaling,armstrong2023scaling}.
For uniformly convex $U$ in $\bbZ^2$, \cite{miller2011fluctuations} showed such a result under general continuous boundary conditions, while \cite{wu2022local} obtained a local central limit theorem for individual field values. 
Under the same conditions, \cite{miller2010universality} proved the convergence of zero contour lines to $SLE(4)$, extending the Gaussian free field case treated in \cite{schramm2009contour} (see also \cite{schramm2013contour}).

The class of Gaussian mixture potentials \eqref{eq:V-GMM} was first considered in the important works \cite{biskup2007phase,biskup2011scaling} as a tractable example of genuine non-convexity. These works showed Gibbs measures can be non-unique even when $\rho$ is supported on two points, but that all zero-slope ergodic Gibbs measures have the Gaussian free field as macroscopic scaling limit when $\rho$ is compactly supported in $(0,\infty)$. 
As in the present work, the idea was to represent $\mu_{G_{\Lambda},V}$ as a mixture of weighted Gaussian free fields, \ie a random conductance model.
\cite{brydges2012fluctuation,ye2019models} proved localization for $V(x)=(1+x^2)^{\alpha}$ with $\alpha\leq 1/2$ by representing it in the form \eqref{eq:V-GMM} and exploiting $\alpha$-stability of (a simple transform of) the associated $\rho$.
\cite{buchholz2021phase} observed that the FKG inequality can be applied to conductances, which was also exploited in \cite{buchholz2023aizenman}. This is the same reason that the FKG inequality is available for us (though we emphasize that the additional ingredient of Gaussian domination is crucial for our approach).
Random conductance models and their associated random walks have been studied extensively in their own right, see e.g. \cite{berger2008anomalous,mathieu2008quenched,biskup2011recent,andres2014invariance,andres2016harnack,andres2020green,dario2021quantitative}.

A related representation when $-V$ (rather than $e^{-V}$ ) is itself a mixture of Gaussian densities was recently introduced in \cite{mukherjee2018identification} to study Fröhlich Polaron, where the Polaron measure was identified in the infinite volume limit as a Gaussian mixture over a tilted Poisson point process on the space of intervals on the real line (see also \cite{betz2022functional,betz2022effective} where this approach has been used). In \cite{bazaes2023effective} we also study Polaron’s effective mass using this point process approach. At one step there, we also apply the FKG-Gaussian correlation inequality to the law of the tilted Poisson process from \cite{mukherjee2018identification}. Thus, while the Gaussian fields in the current work are parametrized by a deformed product measure, for the Polaron one instead uses a deformed Poisson point process. Our previous work \cite{sellke2022almost,betz2023mean} also used the Gaussian correlation inequality to dominate the Polaron and similar path measures by mixtures of Gaussians.
However the details are quite different: there the mixture is used to isolate atypical time intervals, and the FKG inequality plays no role.

\subsection{Notations and Definitions}
\label{subsec:setup}

We write $\cP(\bbR^d)$ for the space of Borel probability measures on $\bbR^d$. 
For $\nu\in \cP(\bbR^d)$ and non-negative $f:\bbR^d\to [0,\infty)$ with $\bbE^{\nu}[f]\in (0,\infty)$, we write $\nu^{(f)}$ for the reweighting $\de\nu^{(f)}(x)\propto f(x)\de\nu(x)$. 
Explicitly,
\begin{equation}
\label{eq:reweight-def}
\nu^{(f)}(A)=\frac{\int_A f(x)\de \nu(x)}{\int_{\bbR^d} f(x)\de\nu(x)}.
\end{equation}
Let $\cS^d_+$ be the set of positive semi-definite $d\times d$ matrices, and write $M_1\preceq_{\PSD} M_2$ if $M_2-M_1\in \cS^d_+$.

Given $\xi,\xi'\in \bbR^d$, we write $\xi\preceq_{\ord} \xi'$ if $\xi_i\leq \xi'_i$ for all $1\leq i\leq d$. 
Also let $\xi\wedge \xi'$ and $\xi\vee\xi'$ respectively denote the coordinate-wise minimum and maximum.
We say a function $f:\bbR^d\to\bbR$ is \emph{coordinate-wise increasing} if $f(\xi)\leq f(\xi')$ whenever $\xi\preceq_{\ord}\xi'$.
We say $f:\bbR^d\to\bbR_+$ is \emph{log-supermodular} if $f(\xi)f(\xi')\leq f(\xi\wedge \xi')f(\xi\vee\xi')$ for all $\xi,\xi'\in\bbR^d$; an absolutely continuous probability measure on $\bbR^d$ is log-supermodular if its density with respect to Lebesgue measure is.
For $\nu,\nu'\in\cP(\bbR^d)$ we write $\nu\preceq_{\stoc}\nu'$ if there exists a coupling of $(x,y)\in\bbR^d\times\bbR^d$ with $x\preceq_{\ord} y$ almost surely, $x\sim\nu$, and $y\sim\nu'$.
These notions are especially relevant for the FKG inequality.

We say a function $f:\bbR^d\to\bbR$ is \emph{symmetric} if $f(x)=f(-x)$ for all $x$, and \emph{quasi-concave} if $f$ is everywhere non-negative and the super level sets $\{x\in \bbR^d~:~f(x)\geq C\}$ are convex for all $C\in\bbR$.
Given probability measures $\mu,\nu\in\cP(\bbR^d)$, we write $\mu\preceq_{\con} \nu$ if $\int f(x)\de\mu(x)\geq \int f(x)\de \nu(x)$ for all symmetric quasi-concave $f:\bbR^d\to\bbR$.
These notions are especially relevant for the Gaussian correlation inequality.

The later parts of this paper use notions of electrical networks on graphs, as explained in \cite[Chapter 2]{lyons2017probability}.
Given a locally finite graph $G$ with $v,w\in V(G)$, and $\xi:E(G)\to\bbR_+$, we write $R_{G,\xi}(v\lrarrow w)$ for the effective resistance between $v,w$ in the electrical network on $G$ with edge resistances $\xi_e^{2}$.
This defines a metric on $V(G)$, and the basic connection with Gaussian free fields is the identity
\begin{equation}
\label{eq:variance-equals-resistance}
    R_{G_{\Lambda},\xi}(v\lrarrow z_{\Lambda})
    =
    \bbE^{\gamma_{G_{\Lambda},\xi}}[\phi(v)^2].
\end{equation}
Here for $\xi\in (0,\infty)^{E(\Lambda)}$, the measure $\gamma_{G_{\Lambda},\xi}(\phi)$ is the wired Gaussian free field on $G_{\Lambda}$ with resistance $\xi_e^{2}$ on edge $e$:
\begin{equation}
\label{eq:GFF-def}
\de\gamma_{G_{\Lambda},\xi}(\phi)
\equiv
\frac{1}{Z_{G, \xi}} 
\exp \bigg(
-
\frac{1}{2}
\sum_{e=\{v,v'\}\in E(G_{\Lambda})} 
\xi_e^{-2}\cdot\big(\phi(v)-\phi(v')\big)^2
\bigg) \delta_0(\phi(z_{\Lambda})) \prod_{v \in V} 
\de \phi(v).
\end{equation}
We write $R_G(v\lrarrow w)$ to denote the same quantity on an unweighted network with $\xi\equiv 1$. We also recall Rayleigh's monotonicity principle: for any $\wh\xi:E(G)\to \bbR_+$ with $\wh\xi|_{E(G_{\Lambda})}=\xi$,
\begin{equation}
\label{eq:Rayleigh-monotonicity}
    R_{G_{\Lambda},\xi}(v\lrarrow z_{\Lambda})
    \leq 
    R_{G,\wh\xi}(v\lrarrow \infty).
\end{equation}
Moreover the connected graph $G$ is transient if and only if $R_{G}(v\lrarrow \infty)<\infty$ for some (or all) $v$.

\section{Correlation Inequalities and Confinement}
\label{sec:FKG-GCI}

This section develops our main tools. After reviewing the FKG and Gaussian correlation inequalities in Subsection~\ref{subsec:correlation-prelims}, we prove the general FKG-Gaussian correlation inequality as Theorem~\ref{thm:FKG-GCI} in Subsection~\ref{subsec:FKG-GCI}.
In Subsection~\ref{subsec:products-1-D} we specialize to densities which are products of $1$-dimensional functions; this still encompasses \eqref{eq:model-def} as explained in Subsection~\ref{subsec:specialization-grad-phi}. Finally Subsection~\ref{subsec:tail-bounds-general} proves Theorem~\ref{thm:convex-comparison}, a general confinement bound for $1$-dimensional projections of such densities.

\subsection{Preliminaries on FKG and Gaussian Correlation Inequalities}
\label{subsec:correlation-prelims}

The first correlation inequality we use is the FKG inequality from \cite{fortuin1971correlation}.

\begin{proposition}[FKG Inequality on $\bbR^d$]
\label{prop:FKG}
    Let $\nu\in\cP(\bbR^d)$ be log-supermodular.
    Then $\nu$ is positively associated in the sense that for any coordinate-wise increasing $f,g:\bbR^d\to\bbR_+$,
    \begin{equation}
    \label{eq:FKG-main}
    \int f(x)g(x)\de \nu(x)
    \geq 
    \int f(x)\de \nu(x)
    \int g(x)\de \nu(x).
    \end{equation}
\end{proposition}

The following lemma is well known, and with the FKG inequality directly implies Proposition~\ref{prop:FKG} below. 

\begin{lemma}[{\cite[Corollary 24.7]{bhattacharya2022special}}]
\label{lem:monotone-coupling-basic}
    Let $\nu_1,\nu_2\in \cP(\bbR^d)$.
    Then $\nu_1\succeq_{\stoc} \nu_2$ if and only if 
    for all coordinate-wise increasing $f:\bbR^d\to\bbR_+$, we have $
    \int f(\xi)\de\nu_1(\xi)
    \geq 
    \int f(\xi)\de\nu_2(\xi)$.
\end{lemma}

\begin{corollary}
\label{cor:FKG-domination}
    Let $\Gamma$ be a log-supermodular probability density on $\bbR^d$, and let $f:\bbR^d\to\bbR_+$ be coordinate-wise increasing (resp. decreasing).
    Then $\Gamma^{(f)}\succeq_{\stoc}\Gamma$ (resp. $\Gamma^{(f)}\preceq_{\stoc}\Gamma$), where we have used the notation \eqref{eq:reweight-def}.
\end{corollary}

We now turn to the Gaussian correlation inequality, shown by Royen in \cite{royen2014simple} (see also \cite{latala2017royen}).

\begin{proposition}
\label{prop:GCI}
    Let $\gamma$ be a centered Gaussian measure on $\bbR^d$. For any symmetric quasi-concave $f_1,f_2,\dots,f_n$:
    \begin{equation}
    \label{eq:GCI}
    \bbE^{x\sim \gamma}\lt[\prod_{j=1}^m f_j(x)\rt]\cdot
    \bbE^{x\sim \gamma}\lt[\prod_{k=m+1}^{n} f_k(x)\rt]
    \leq
    \bbE^{x\sim \gamma}\lt[\prod_{i=1}^n f_i(x)\rt].
    \end{equation}
\end{proposition}

\begin{proof}
    By Fubini's theorem it suffices to assume each $f_i=1_{K_i}$ is the indicator of a symmetric convex set. 
    As such sets are closed under intersection we are immediately reduced to the case $(m,n)=(1,2)$. This is the usual Gaussian correlation inequality proved in \cite{royen2014simple}.
\end{proof}

\begin{corollary}
\label{cor:GCI}
    Let $\gamma,\nu$ be symmetric probability measures on $\bbR^d$ with $\gamma$ Gaussian and $\de\nu/\de\gamma$ a finite product of symmetric quasi-concave functions. Then $\nu\preceq_{\con} \gamma$.
\end{corollary}

\begin{corollary}
\label{cor:tilt-monotonicity}
    Let $\gamma_i\propto e^{-\la x,Q_i x\ra/2}\de x$ for $i\in \{1,2\}$ be centered Gaussian measures on $\bbR^d$, with $Q_1\succeq_{\PSD} Q_2$. 
    Then for any symmetric quasi-concave $f:\bbR^d\to\bbR$, we have $\int f(x)\de\gamma_1(x)\geq \int f(x)\de\gamma_2(x)$.
\end{corollary}

\begin{proof}
    This was shown in \cite[Corollary 3]{anderson1955integral}, and is also immediate from Corollary~\ref{cor:GCI}.
\end{proof}

\subsection{The FKG-Gaussian Correlation Inequality}
\label{subsec:FKG-GCI}

Our main new tool, Theorem~\ref{thm:FKG-GCI} below, is a combination of the preceding inequalities.
It extends the Gaussian correlation inequality to a special class of distributions we call log-supermodular centered (LSMC) Gaussian mixtures.

\begin{definition}
\label{def:LSM-Gaussian-mixture}
    Let $\nu$ be a log-supermodular probability measure on $\bbR_+^d$, and for each $\xi\in\bbR_+^d$ let $\de\gamma_{\xi}(x)\propto e^{-\la x,F(\xi) x\ra/2}$ be a centered Gaussian measure on $\bbR^n$, for $F:\bbR^d_+\to\cS^n_+$ which is decreasing from $\preceq_{\ord}$ to $\preceq_{\PSD}$. 
    We define the \textbf{LSMC Gaussian mixture}
    \begin{equation}
    \label{eq:LSMC}
    \Gamma_{\nu,F}
    =
    \int
    \gamma_{\xi}
    ~\de \nu(\xi)
    \in\cP(\bbR^n).
    \end{equation}
\end{definition}

\begin{theorem}[FKG-Gaussian correlation inequality]
\label{thm:FKG-GCI}
    Let $\Gamma=\Gamma_{\nu,F}$ be an LSMC Gaussian mixture. Then Proposition~\ref{prop:GCI} and Corollary~\ref{cor:GCI} hold with $\Gamma$ in place of $\gamma$.
\end{theorem}

\begin{proof}
    Let $\nu$ have density $F$.
    As in the proof of Proposition~\ref{prop:GCI}, it suffices to handle the case of two symmetric quasi-concave functions $f_1,f_2:\bbR^n\to\bbR_+$, i.e. to show $\int f_2(x) \de\Gamma^{(f_1)}(x)\geq \int f_2(x) \de\Gamma(x)$.
    To do so, we write
    \begin{equation}
    \label{eq:Gamma-f1}
    \Gamma^{(f_1)}
    =
    \int
    \gamma_{\xi}^{(f_1)}
    \de \wt\nu(\xi)
    \end{equation} 
    where it is easily checked $\de\wt\nu(\xi)/\de\nu(\xi)\propto \int f_1(x)\de\gamma_{\xi}(x)$. 
    Since $F$ is assumed decreasing, Corollary~\ref{cor:tilt-monotonicity} implies this Radon--Nikodym derivative is coordinate-wise decreasing in $\xi$.
    Therefore the FKG inequality in the form of Corollary~\ref{cor:FKG-domination} implies $\wt\nu\preceq_{\stoc} \nu$.
    In particular there exists a coupling $(\xi,\wt \xi)$ such that $\xi$ has law $\nu$, $\wt\xi$ has law $\wt\nu$, and $\wt\xi\preceq_{\ord}\xi$ almost surely. 
    Then for any symmetric quasi-concave $f_2:\bbR^n\to\bbR^+$, we have $\int f_2(x)\de \gamma_{\xi}\leq \int f_2(x)\de\gamma_{\wt\xi}$ almost surely by Corollary~\ref{cor:tilt-monotonicity}.
    Hence
    \begin{align}
    \label{eq:FKG-GCI-computation}
    \int f_2(x) \de\Gamma^{(f_1)}(x)
    &=
    \iint f_2(x) \de\gamma_{\wt\xi}^{(f_1)}(x) \de \wt\nu(\wt\xi)
    \geq 
    \iint f_2(x) \de\gamma_{\wt\xi}(x) \de \wt\nu(\wt\xi)
    \\
    \notag
    &\geq 
    \iint f_2(x) \de\gamma_{\xi}(x) \de \nu(\xi)
    =
    \int f_2(x) \de\Gamma(x)
    \qedhere
    \end{align}
\end{proof}

In the special case $d=1$, Theorem~\ref{thm:FKG-GCI} extends the Gaussian correlation inequality to base measures which are mixtures of centered Gaussians with $\preceq_{\PSD}$-totally ordered covariances.
Conversely, it is easy to show that symmetric convex sets can be negatively correlated relative to mixtures of two incomparable centered Gaussians, even on $\bbR^2$.

\begin{remark}
\label{rem:no-GCI}
    For the results in this paper, the full power of the Gaussian correlation inequality can be avoided.
    The reason is that we will apply comparison results such as Theorem~\ref{thm:domination-of-products} only to $1$-dimensional projections (for the tails of individual field values). Thus within its proof we only need to handle $f_2$ of the form $f_2(x)=1_{|\la v,x\ra|\leq C}$.
    This reduces to the Gaussian correlation inequality in the case $K_2=K_2(v,C)=\{x:|\la v,x\ra|\leq C\}$, which is much easier and shown in \cite{khatri1967certain,sidak1967rectangular}.
    Our subsequent arguments go through unchanged modulo redefining $\preceq_{\con}$ to involve only these \emph{one-dimensional} symmetric quasi-concave test functions.
    On the other hand, the full statement of Theorem~\ref{thm:FKG-GCI} (and Theorem~\ref{thm:domination-of-products} below) may be more powerful in other situations, for example in a closer examination of the maximum field value, and is in our opinion more elegant than the weaker variant outlined above.
\end{remark}

\subsection{Products of $1$-Dimensional Functions}
\label{subsec:products-1-D}

Next we specialize to the case when $\Gamma$ is a product of one-dimensional functions, each a mixture of centered Gaussian densities.
We index these by $e\in E$ since we later take $E=E(G_{\Lambda})$ (thus $E$ replaces $d$).
Given a family $(U_e)_{e\in E}:\bbR\to\bbR_+$ and a spanning set of vectors $(y_e)_{e\in E}$ in $\bbR^n$, define $\Gamma_{\vec V,\vec y}\in\cP(\bbR^n)$ and $F:\bbR^E_+\to\bbS^n_+$ as:
\begin{align}
\label{eq:product-of-functions}
    \de\Gamma_{\vec U,\vec y}(x)
    &\propto
    \exp\Big(-\sum_{e\in E} U_e(\la x,y_e\ra)\Big)\de x,
    \\
\label{eq:F-abstract-form}
    F(\xi)&=\sum_{e\in E} \xi_e^{-2} y_e^{\otimes 2}.
\end{align}
The following proposition, an easy (omitted) computation, gives the Gaussian mixture representation of such a product.

\begin{proposition}
\label{prop:product-of-functions}
    For $e\in E$, let $V_e:\bbR\to\bbR_+$ take the form \eqref{eq:V-GMM} with associated $\rho_e\in\cP(\bbR_+)$.
    With $\Gamma_{\vec V,\vec y},F$ as above, we have $\Gamma_{\vec V,\vec y}=\Gamma_{\nu,F}$ (recall \eqref{eq:LSMC}) for 
    \[
    \de\nu(\xi)\propto \Big(\det(F(\xi))^{-1/2}\prod_{e\in E} \xi_e^{-1} \Big) \prod_{e\in E}
    \de\rho_e(\xi_e).
    \]
\end{proposition}

\begin{theorem}
\label{thm:domination-of-products}
    In the setting of Proposition~\ref{prop:product-of-functions}, let $U_e:\bbR\to\bbR_+$ be such that $U_e-V_e$ is symmetric and increasing on $\bbR_+$ for each $e\in E$.
    Then:
    \begin{enumerate}[label = {(\alph*)}]
    \item 
    \label{it:product-is-LSM}
    $\nu(\xi)$ from Proposition~\ref{prop:product-of-functions} is log-supermodular, i.e. $\Gamma_{\vec V,\vec y}$ is an LSMC Gaussian mixture.
    \item 
    \label{it:product-of-rho-domination}
    $\nu(\xi)\preceq_{\stoc}\prod_{e\in E} \rho_e(\xi_e)$.
    \item 
    \label{it:Gaussian-system-domination}
    $\Gamma_{\vec U,\vec y}
    \preceq_{\con} 
    \Gamma_{\vec V,\vec y}
    \preceq_{\con}
    \int
    \gamma_{\xi}
    \prod_{e\in E}\de\rho_e(\xi_e).$
    \end{enumerate}
\end{theorem}

\begin{proof}
    For \ref{it:product-is-LSM}, since $G$ is linear and coordinate-wise-to-PSD increasing in $\xi$, it suffices to show
    \begin{equation}
    \label{eq:det-ABC}
    \det(A+B+C)\det(A)\leq \det(A+B)\det(A+C),
    \quad
    \forall
    A,B,C\in\cS^n_+.
    \end{equation}
    This is the Gaussian correlation inequality with $\de\gamma(x)\propto e^{-\la x,Ax\ra/2}$ and $f_1(x)=e^{-\la x,Bx\ra/2}$, $f_2(x)=e^{-\la x,Cx\ra/2}$ . 

    To prove \ref{it:product-of-rho-domination} we use FKG with base measure $\prod_{e\in E} \rho_e(\xi_e)$; for this it suffices to show $\xi\mapsto\det(F(\xi))^{-1/2}\prod_{e\in E} \xi_e^{-1}$ is coordinate-wise decreasing.
    Indeed, monotonicity in each individual coordinate $\xi_i$ reduces to \eqref{eq:det-ABC}.
    
    The first part of claim~\ref{it:Gaussian-system-domination} is just Theorem~\ref{thm:FKG-GCI}, which is applicable thanks to claim~\ref{it:product-is-LSM}.
    The second part follows from claim~\ref{it:product-of-rho-domination} since $\gamma_{\xi}\preceq_{\con} \gamma_{\xi'}$ whenever $\xi\succeq_{\ord}\xi'$.
\end{proof}

Theorem~\ref{thm:domination-of-products}\ref{it:Gaussian-system-domination} gives a natural method to prove confinement of $\Gamma_{\vec U,\vec y}$. Namely if suitable $\vec V$ can be found, the right-hand side is a completely explicit mixture of centered Gaussians. This is how we will prove Theorem~\ref{thm:main}.

\begin{remark}
    In \eqref{eq:product-of-functions} and \eqref{eq:F-abstract-form}, one could more generally replace each $y_e^{\otimes 2}$ with a positive semi-definite matrix $M_e\in\cS^n_+$, such that $\sum_e M_e$ is strictly positive definite. 
    Then $\la x,y_e\ra$ becomes $\sqrt{\la x,M_e x\ra}$. Since each decreasing function $U_e-V_e$ applied to $\sqrt{\la x,M_e x\ra}$ is still quasi-concave on $\bbR^n$, Theorem~\ref{thm:domination-of-products} still holds.
    Using this, all our results (and their proofs) extend to vector-valued fields $\phi:V(G)\to \bbR^k$, with e.g. radial potentials $U_e(\|\phi(v)-\phi(v')\|)$.
\end{remark}

\subsection{Specialization to $\nabla\phi$ Surfaces}
\label{subsec:specialization-grad-phi}

The $\nabla\phi$ Gibbs measure \eqref{eq:model-def} with potential $V$ as in \eqref{eq:V-GMM} is a LSMC Gaussian mixture, and falls under the purview of Theorem~\ref{thm:domination-of-products}.
Here $\gamma_{G_{\Lambda},\xi}(\phi)$ is, as defined in \eqref{eq:GFF-def}, the wired Gaussian free field on $G_{\Lambda}$ with resistances $\xi_e^{2}$.
Thus $F(\xi)$ is the $\Lambda\times\Lambda$ matrix
\begin{equation}
\label{eq:F-graph}
F(\xi)_{v,v'}
=
\begin{cases}
\sum\limits_{v''\in \Lambda\cup \{z_{\Lambda}\}}
\xi_{v,v''}^{-2},\quad v=v'
\\
\quad\quad\quad\quad
-\xi_{v,v'}^{-2},\quad v\neq v'.
\end{cases}
\end{equation}
Correspondingly, $y_{\{v,z_{\Lambda}\}}$ has entry $1$ at $v$ with the remaining entries zero, while $y_{\{v,v'\}}$ has entry $1$ at $v$ and $-1$ at $v'$ (or vice-versa) when $v,v'\in\Lambda$.
The lemma below is then a direct specialization of Theorem~\ref{thm:domination-of-products}\ref{it:Gaussian-system-domination}.

\begin{lemma}
\label{lem:stiffening}
    Suppose $\vec U=V+\vec W$ for $V$ as in \eqref{eq:V-GMM} and symmetric quasi-concave $\vec W$. 
    Then
    \[
    \mu_{G_{\Lambda},\vec U}
    \preceq_{\con} 
    \mu_{G_{\Lambda},V}
    \preceq_{\con}
    \int\gamma_{\Lambda,\xi} ~\de\rho^{\otimes E(G_{\Lambda})}(\xi).
    \]
\end{lemma}

We note that \cite[Section IV]{brascamp1975statistical} gives an explicit example with convex $U$ where increasing the interaction strengths (``stiffening the springs'') strictly increases the variance of $\phi(v)$.
Lemma~\ref{lem:stiffening} says this is impossible when the pre-perturbation potential takes the form \eqref{eq:V-GMM}.

\subsection{Tail Bounds for $1$-Dimensional Projections}
\label{subsec:tail-bounds-general}

We continue the study of linear $F$ as in \eqref{eq:F-abstract-form}, and show that tail decay for the $\xi_e$ implies confinement for the Gaussian mixture $\Gamma$.
Let $\Phi:\bbR_+\to\bbR_+$ be a convex increasing function with $\lim_{x\to\infty}\Phi(x)=\infty$, and define
\begin{equation}
\label{eq:Q-def}
    Q_{\xi}(y)=\bbE^{x\sim \gamma_{\xi}}[\la x,y\ra^2].
\end{equation}
Theorem~\ref{thm:convex-comparison} essentially shows that bounded $\Phi$-moments of each individual $\xi_e$ implies similar bounds for $Q_{\xi}(y)$, whenever $Q_1(y)$ is bounded.
Here and below we write $1$ in place of $\xi$ to indicate that $\xi_e=1$ for all $e$. 
Note that if $y$ corresponds to the value $\phi(v)$, boundedness of $Q_1(y)$ corresponds to transience of $G$.

\begin{theorem}
\label{thm:convex-comparison}
    For any $y\in\bbR^n$ and $\nu\in\cP(\bbR^d_+)$, we have
    $
    \bbE^{\xi\sim\nu}[\Phi(Q_{\xi}(y))]
    \leq 
    \sup_{e\in E}
    \bbE^{\xi\sim\nu}\Phi(\xi_e^2 Q_{1}(y))
    $.
\end{theorem}

\begin{proof}
    Note that $F$ is almost surely strictly positive definite since $\Phi$ tends to infinity.
    For all $\xi$,
    \[
    Q_{\xi}(y)
    =
    \la y,F(\xi)^{-1}y\ra
    =
    \min_{w\neq 0}\frac{\la y,w\ra^2}{\la w,F(\xi)w\ra}
    \leq 
    \frac{\la y,F(1)^{-1} y\ra^2}{\la y,F(1)^{-1} F(\xi) F(1)^{-1} y\ra}
    =
    \frac{\big(\sum_e\la y,F(1)^{-1} M_e F(1)^{-1} y\ra\big)^2}{\sum_e
    \xi_e^{-2}
    \la y,F(1)^{-1} M_e F(1)^{-1} y\ra}
    .
    \]
    By Cauchy--Schwarz, the last expression is at most $\sum_e \xi_e^2 \la y,F(1)^{-1} M_e F(1)^{-1} y\ra$.
    Since $\Phi$ is increasing and convex, 
    \[
    \bbE \Phi(Q_{\xi}(y))
    \leq 
    \bbE \Phi\lt(
    \sum_e \xi_e^2 \la y,F(1)^{-1} M_e F(1)^{-1} y\ra
    \rt)
    \leq 
    \sup_e 
    \bbE\Phi(\xi_e^2 \la y,F(1)^{-1} y\ra).
    \qedhere
    \]
\end{proof}

The argument above is similar to that of \cite{brydges1990grad} for the case $\Phi(x)=x^{\alpha}$. Actually for $\nabla\phi$ models it can be formulated quite intuitively: one fixes the minimal-energy $v\to z_{\Lambda}$ unit flow for the \textbf{unweighted} graph $G_{\Lambda}$, and uses Jensen's inequality to bound the $\Phi$-moment of its $\xi_e^{-2}$-\textbf{weighted} energy (an upper bound for effective resistance).

The next propositions characterize polynomial and stretched exponential decay in terms of test functions, which will be convenient for applying Theorem~\ref{thm:convex-comparison}. We omit the easy proofs.

\begin{proposition}
\label{prop:power-tail-convex}
    For any $\alpha>2$ the following are equivalent (up to the values $C_1,C_2$) for $\mu\in\cP(\bbR_+)$:
    \begin{enumerate}[label = {(\alph*)}]
        \item $\bbP^{x\sim\mu}[x\geq t]\leq C_1 t^{-\frac{\alpha}{2}}$.
        \item $\bbE^{x\sim\mu}[(x-t)_+]\leq C_2 t^{1-\frac{\alpha}{2}}$ for all $t\geq 0$.
    \end{enumerate}
    Moreover in either case, if $Z$ is an independent standard Gaussian, then $\bbP[Z\sqrt{x}\geq t]\leq C_3 t^{-\alpha}$.
\end{proposition}

\begin{proposition}
\label{prop:stretched-tail-convex}
    Fix $\beta\in (0,2)$ and $\wt\beta=\frac{\beta}{2-\beta}$. The following are equivalent (up to $C_1,C_2$) for $\mu\in\cP(\bbR_+)$:
    \begin{enumerate}[label = {(\alph*)}]
        \item $\bbP^{x\sim\mu}[x\geq t]\leq C_1 e^{-t^{\wt\beta}/C_1}$ for all $t>0$.
        \item $\bbE^{x\sim\mu}[e^{x^{\wt\beta}/C_2}]\leq C_2$.
    \end{enumerate}
    Moreover in either case, if $Z$ is an independent standard Gaussian, then $\bbP[Z\sqrt{x}\geq t]\leq C_3 e^{-t^{\beta}/C_3}$.
\end{proposition}

\section{Proofs of Main Results}
\label{sec:main-proofs}

In Subsection~\ref{subsec:explicit-mixture-reps} we explicitly construct $V$ as in  \eqref{eq:V-GMM} for various classes of monotone functions.
We then prove the main results of the paper using the ideas of the previous section, especially Theorems~\ref{thm:domination-of-products} and \ref{thm:convex-comparison}.

\subsection{Gaussian Mixture Representations of Monotone Functions}
\label{subsec:explicit-mixture-reps}

Here we give the crucial constructions of Gaussian mixtures $V$ for monotone $U$. 
Although our main results consider only $\alpha\in \{\eps\}\cup (2,\infty)$, the construction below works for all $\alpha,\eps>0$.
Note that $V$ depends only on $(\alpha,\eps)$ or $(\beta,\eps)$.

\begin{proposition}
\label{prop:monotone-function-representation}
    For any $\alpha,\eps>0$ there exists $V=V_{\alpha,\eps}$ of the form \eqref{eq:V-GMM} such that $W=U-V$ is quasi-concave for all $(\alpha,\eps)$-monotone $U$.
    Further, the measure $\rho=\rho_{\alpha,\eps}$ may be taken to have sub-$\alpha$ tails.
\end{proposition}

\begin{proof}
    It suffices to construct $V$ of the form \eqref{eq:V-GMM} with $V'(x)\leq \min\big(\eps x,\frac{\alpha+1}{x}\big)$. 
    We take $\rho$ to have density proportional to $\kappa^{-\alpha-1}\cdot 1_{\kappa\geq A}$ with $A=1+\eps^{-1/2}$.
    Since the ratio between integrands below is decreasing in $\kappa$, 
    \begin{align*}
    V'(x)
    &=
    \frac{x\int_A^{\infty} 
    \kappa^{-\alpha-4}
    e^{-x^2/2\kappa^2}
    \de \kappa}
    {
    \int_A^{\infty} 
    \kappa^{-\alpha-2}
    e^{-x^2/2\kappa^2}
    \de \kappa}
    \leq 
    \frac{x}{A^2}
    \leq
    \eps x,
    \end{align*}
    proving the first desired upper bound.
    By similar monotonicity, another upper bound can be obtained by setting $A=0$ in the integrals above. 
    However with 
    \[
    f(x)=\int_0^{\infty} 
    \kappa^{-\alpha-2}
    e^{-x^2/2\kappa^2}
    \de\kappa
    \]
    it is easy to see by substitution that $f(x)=f(1)x^{-\alpha-1}$ for all $x>0$.
    We conclude that  
    \[
    V'(x)
    =
    \frac{x\int_A^{\infty} 
    \kappa^{-\alpha-4}
    e^{-x^2/2\kappa^2}
    \de \kappa}
    {
    \int_A^{\infty} 
    \kappa^{-\alpha-2}
    e^{-x^2/2\kappa^2}
    \de \kappa}
    \leq
    \frac{x\int_0^{\infty} 
    \kappa^{-\alpha-4}
    e^{-x^2/2\kappa^2}\de\kappa}
    {\int_0^{\infty} 
    \kappa^{-\alpha-2}
    e^{-x^2/2\kappa^2}\de\kappa}
    =
    -\frac{f'(x)}{f(x)}
    =
    \frac{\alpha+1}{x}.
    \qedhere
    \]
\end{proof}

We next turn to the polynomially monotone case.
Given $K,\beta,\eps>0$, let $\wh V_{\beta,K}(x)=(1+(x/K)^2)^{\beta/2}$.

\begin{proposition}
\label{prop:V-beta-K-GMM}
    For $\beta\in (0,2)$, $\wh V_{\beta,K}$ takes the form \eqref{eq:V-GMM}, and the associated $\wt\rho$ has stretched sub-$\frac{2\beta}{2-\beta}$ tail.
\end{proposition}

\begin{proof}
    We set $K=1$ by rescaling.
    An explicit formula for $f(\kappa)\equiv\de\wt\rho(\kappa)/\de\kappa$ is essentially given in \cite{ye2019models}. 
    In particular \cite[Proposition A.8]{ye2019models} gives an asymptotic description of the density as $\kappa\to\infty$, where his $\alpha\in (0,1)$ is our $\beta/2$, and we set his $\beta$ to $1$.
    (Much of \cite{ye2019models} assumes $\alpha<1/2$, but the cited result just uses \cite[Lemma 1.1]{hawkes1971lower} which holds for all $\alpha\in (0,1)$.)
    Dropping the factor $e^{-\kappa}<1$ written in \cite[Equation (1.6)]{ye2019models} (note $\kappa$ is parametrized differently from \ref{eq:V-GMM}), the desired bound easily follows from \cite[Proposition A.8]{ye2019models}.
\end{proof}

Having seen that $e^{-\wh V_{\beta,K}}$ has a Gaussian mixture representation, we now lower bound its derivative.

\begin{proposition}
\label{prop:power-growth-GMM}
    For any $\beta\in (0,2)$ and $\eps>0$, and $K$ sufficiently large depending on $(\beta,\eps)$, we have $\wh V_{\beta,K}'(x)\leq \eps\,\min(x,x^{\beta-1})$ for all $x>0$. 
\end{proposition}

\begin{proof}
    Since $\wh V_{\beta,K}'(x)=\frac{\beta x}{K^2}\lt(1+(x/K)^2\rt)^{\frac{\beta}{2}-1}$ the upper bound $\eps x$ is clear. 
    For the latter bound, we write
    \[
    \wh V_{\beta,K}'(x)
    =
    \frac{\beta}{K^{\beta}x^{1-\beta}}
    \cdot 
    \lt(\frac{x^2}{x^2+K^2}\rt)^{1-\frac{\beta}{2}}
    \leq 
    \frac{\beta}{K^{\beta}x^{1-\beta}}.
    \qedhere
    \]
\end{proof}

\subsection{Proofs for $\nabla\phi$ Models with Wired Boundary Conditions}
\label{subsec:localization-wired-proofs}

Below we always fix $v\in \Lambda\subseteq G$. 
We will repeatedly use the fact that if $\mu_i\preceq_{\con}\nu_i$ for $i\geq 1$, then all tail bounds for $\phi(v)$ (tightness, polynomial tails, stretched exponential tails) under $\nu_i$ directly transfer to $\mu_i$.
Since the functions $V$ constructed in the previous subsection do not depend on $U$, the measures $\rho_e=\rho$ will not depend on $e$ below.

\begin{proof}[Proof of Theorem~\ref{thm:main}]
Let $\rho$ be as in Proposition~\ref{prop:monotone-function-representation}.
In light of Lemma~\ref{lem:stiffening}, it remains to prove tightness of $\phi(v)$ under the mixture of Gaussian free fields $\int\gamma_{\Lambda,\xi} \,\de\rho^{\otimes E(G_{\Lambda})}(\xi)$.
Fix $p<1$ such that $p$-bond percolation on $G$ almost surely has a transient infinite cluster, and let $C_{\rho,p}<\infty$ be such that $\rho([0,C_{\rho,p}])\geq p$.

Let $\wh\xi\in (0,\infty)^{E(G)}$ have law $\rho^{\otimes E(G)}$ and let $\xi=\wh\xi|_{E(G_{\Lambda})}\sim\rho^{\otimes E(G_{\Lambda})}$ be the restriction to $E(G_{\Lambda})\subseteq E(G)$. 
Let $H_{\wh\xi}\subseteq G$ be the subgraph consisting of edges with resistance at most $C_{\rho,p}^2$:
\[
E(H_{\wh\xi})\equiv\{e\in E(G)~:~\wh\xi_e\leq C_{\rho,p}\}.
\]
By definition $H_{\wh\xi}$ has a transient infinite cluster $\cC$ almost surely.
Let $w\in\cC$ be a closest point to $v=(0,0)$; then $R_{H_{\wh\xi}}(w\lrarrow \infty)<\infty$ by transience of $\cC$.
Since $R_{G,\wh\xi}(v\lrarrow w)$ is trivially finite, we have
\begin{equation}
\label{eq:resistance-computation}
R_{G,\wh\xi}(v\lrarrow \infty)
\leq 
R_{G,\wh\xi}(v\lrarrow w)
+
R_{G,\wh\xi}(w\lrarrow \infty)
\leq 
R_{G,\wh\xi}(v\lrarrow w)
+
C_{\rho,p}^{2}
R_{H_{\wh\xi}}(w\lrarrow \infty)
<
\infty.
\end{equation}
Recall from \eqref{eq:Rayleigh-monotonicity} that $R_{G_{\Lambda},\xi}(v\lrarrow z_{\Lambda})\leq R_{G,\wh\xi}(v\lrarrow \infty)$ almost surely.
In particular as $\Lambda$ varies $R_{G_{\Lambda},\xi}(v\lrarrow z_{\Lambda})$ is uniformly stochastically dominated by the almost surely finite $R_{G,\wh\xi}(v\lrarrow \infty)$. Hence the random variables $R_{G_{\Lambda},\xi}(v\lrarrow z_{\Lambda})$ are tight.
By \eqref{eq:variance-equals-resistance} the $\int\gamma_{\Lambda,\xi} \de\rho^{\otimes E(G_{\Lambda})}(\xi)$-law of $\phi(v)$ is a mixture of centered Gaussians with variance $R_{G_{\Lambda},\xi}(v\lrarrow z_{\Lambda})$, so we conclude the desired tightness of $\phi(v)$.
\end{proof}

\begin{proof}[Proof of Theorem~\ref{thm:Z2}]
    Let $L=\diam(\Lambda)$ and let $\Lambda_L=[-L,\dots,L]^2\supseteq \Lambda$.
    Writing $\vec U=V+\vec W$ as usual,
    \begin{equation}
    \label{eq:Z2-comparison}
    \mu_{\bbZ^2_{\Lambda},\vec U}
    \preceq_{\con}
    \mu_{\bbZ^2_{\Lambda},V}
    \preceq_{\con}
    \mu_{\bbZ^2_{\Lambda_L},V}
    \preceq_{\con}
    \int \gamma_{\Lambda_L,\xi}
    ~\de\rho^{\otimes E(G_{\Lambda_L})}(\xi).
    \end{equation}
    Here the first and third comparisons come from Theorem~\ref{thm:domination-of-products}\ref{it:Gaussian-system-domination}, and the second from Corollary~\ref{cor:tilt-monotonicity} since $\mu_{\bbZ^2_{\Lambda},V}$ is just the centered Gaussian process $\mu_{\bbZ^2_{\Lambda_L},V}$ conditioned to lie in a linear subspace.

    We now proceed as in Theorem~\ref{thm:main} with $p>p_c(\bbZ^2)=1/2$.
    The same proof works, except the infinite cluster of $H_{\wh\xi}$ is not transient in \eqref{eq:resistance-computation}.
    Let $\cC$ be the infinite cluster of $H_{\wh\xi}$ for an extended $\wh\xi\sim\rho^{\otimes E(G)}$ as before, and $w\in \cC$ a closest point to $v$.
    Let $\cC_L$ be the connected component of $\cC\cap E(\Lambda_L)$ containing $w$.
    Then
    \begin{equation}
    \label{eq:resistances-Z2}
    R_{G_{\Lambda_L},\xi}(v\lrarrow z_{\Lambda_L})
    \leq 
    R_{G_{\Lambda_L},\xi}(v\lrarrow w)
    +
    R_{G_{\Lambda_L},\xi}(w\lrarrow z_{\Lambda_L})
    \leq 
    R_{G_{\Lambda_L},\xi}(v\lrarrow w)
    +
    C_{\rho,p}^{2} 
    R_{\cC_L}(w\lrarrow z_{\Lambda_L})
    .
    \end{equation}
    It is shown in \cite[Proposition 4.3]{boivin2013existence} (see also \cite{abe2015effective}) that almost surely\footnote{As stated the result is for a radius $L$ ball $\cB_L$ in the $\cC$-graph distance. Since $\cB_{L^2}\supseteq\cC_L$, \eqref{eq:Z2-cited-bound} follows with slightly larger constant $C$.},
    \begin{equation}
    \label{eq:Z2-cited-bound}
    \limsup_{L\to\infty}
    \frac{R_{\cC_L}(w\lrarrow z_{\Lambda_L})}{\log L}
    \leq C<\infty.
    \end{equation}
    Since the $\bbZ^2$-graph distance $d(v,w)$ is finite, the shortest path from $v$ to $w$ in $\bbZ^2$ lies in $\Lambda_L$ with probability $1-o_L(1)$.
    It follows that the laws of $R_{G_{\Lambda_L},\xi}(v\lrarrow w)$ are tight as $L$ varies.
    Combining completes the proof.
\end{proof}

\begin{proof}[Proof of Theorem~\ref{thm:moment-bounds}]
    With $V$ as in Proposition~\ref{prop:monotone-function-representation} but now with $\alpha>2$, using Theorem~\ref{thm:domination-of-products} yields as before
    \begin{equation}
    \label{eq:domination-as-before}
    \mu_{G_{\Lambda},\vec U}
    \preceq_{\con}
    \mu_{G_{\Lambda},V}
    \preceq_{\con}
    \int
    \gamma_{\Lambda,\xi}
    ~\de\rho^{\otimes E(G_{\Lambda})}(\xi).
    \end{equation}
    We apply Theorem~\ref{thm:convex-comparison} with $\nu=\rho^{\otimes E(G_{\Lambda})}$ and $\Phi_t=(x-t)_+$.
    Abusing the notation \eqref{eq:Q-def} so  $Q_{\xi}(v)=\bbE^{\gamma_{\xi}} [\phi(v)^2]$,
    \begin{equation}
    \label{eq:convex-tail-finish-moments}
    \bbE^{\xi\sim\nu}[\Phi_t(Q_{\xi}(v))]
    \leq 
    \bbE^{\xi_e\sim\rho}
    [\Phi_t(\xi_e^2 Q_1(v))]
    \leq 
    CQ_1(v)^{1-\frac{\alpha}{2}}
    t^{-\alpha/2}.
    \end{equation}
    Here the last step follows by applying Proposition~\ref{prop:power-tail-convex} to $\xi_e^2$ (which has sub-$\alpha/2$ tail since $\xi_e$ have uniformly sub-$\alpha$ tails) and scaling.
    Note that $Q_1(v)=R_{G_{\Lambda}}(v\lrarrow z_{\Lambda})$.
    Applying Proposition~\ref{prop:power-tail-convex} in the opposite direction, the previous display implies $Q_{\xi}(v)$ also has sub-$\alpha/2$ tail.
    The last implication of Proposition~\ref{prop:power-tail-convex} now completes the proof.
\end{proof}

\begin{proof}[Proof of Theorem~\ref{thm:stretched-exponential-bounds}]
    If $\beta=2$, just take $V(x)=\eps x^2/2$. Then even Corollary~\ref{cor:GCI} shows $\mu_{G_{\Lambda},\vec U}\preceq_{\con}\mu_{G_{\Lambda},V}$ is dominated by a constant multiple of an unweighted Gaussian free field as desired.

    For $\beta\in (0,2)$, set $\wt\beta=\frac{\beta}{2-\beta}$.
    For $c_*=c_*(\beta,Q_1(v))$ small enough and $K=K(\beta,Q_1(v),c_*)$ large enough,
    $
    \Phi_{\wt\beta}(x)
    \equiv
    \max\Big(
    K,e^{c_*|x|^{\wt\beta}}
    \Big)
    $
    is convex. Similarly to \eqref{eq:convex-tail-finish-moments}, with $\nu=\rho^{\otimes E(G_{\Lambda})}$ for $\rho$ now as in Propositions~\ref{prop:V-beta-K-GMM} and \ref{prop:power-growth-GMM},
    \begin{equation}
    \label{eq:convex-tail-finish-stretched}
    \bbE^{\xi\sim\nu}
    [\Phi_{\wt\beta}(Q_{\xi}(v))]
    \leq 
    \bbE^{\xi_e\sim\rho}
    [\Phi_{\wt\beta}(\xi_e^2 Q_1(v))]
    .
    \end{equation} 
    Proposition~\ref{prop:V-beta-K-GMM} ensures that $\xi_e^2$ has stretched sub-$\wt\beta$ tail.
    Hence by Proposition~\ref{prop:stretched-tail-convex}, the right-hand side of \eqref{eq:convex-tail-finish-stretched} is finite for small $c_*$.
    Applying Proposition~\ref{prop:stretched-tail-convex} in reverse, $Q_{\xi}(v)$ thus has stretched sub-$\wt\beta$ tail.
    Since the $\gamma_{\xi}$-law of $\phi(v)$ is centered Gaussian with variance $Q_{\xi}(v)$, the last implication of Proposition~\ref{prop:stretched-tail-convex} completes the proof.
\end{proof}

\subsection{Proofs for Strongly Transient Graphs and Maximum Field Values}

\begin{proof}[Proof of Theorem~\ref{thm:moment-bounds-Zd}]
    Strong transience of regular trees is obvious. We give the proof of both parts in the special case of $\bbZ^d$, the extension to general strongly transient graphs being identical. 
    After partitioning $\bbZ^d$ into $2d$ edge-disjoint transient subgraphs containing $\vec 0$, we apply the proof of Theorem~\ref{thm:moment-bounds} to each. This upper-bounds $Q_{\xi}(\vec 0)$ by the minimum of $2d$ independent sub-$\alpha/2$ variables, showing $Q_{\xi}(\vec 0)$ is sub-$d\alpha$, so $\phi(\vec 0)$ is sub-$2d\alpha$ by Proposition~\ref{prop:power-tail-convex}.
    
    We partition $\bbZ^d$ as follows. 
    For each $1\leq i\leq d$ write $\iota_i\in\bbZ^d$ for the $i$-th standard basis vector. 
    Given also $j\in\bbZ^d$ let $j\equiv \hat j$ with $1\leq \hat j\leq d$.
    Let
    $\iota_{i,j}=\iota_i$ if $i\leq \hat j$, and $\iota_{i,j}=-\iota_i$ if $i>\hat j$. 
    For $\eta\in \{\pm 1\}$, define the length $d$ path in $\bbZ^d$:
    \[
    P_{i,\eta}=
    \Big(
    \vec 0, 
    \eta \iota_i,
    \eta \big(\iota_i+\iota_{i,i+1}\big),
    \dots,
    \eta \big(\iota_i+\iota_{i,i+1}+\dots+\iota_{i,i+d-1}\big)
    \Big).
    \]
    For example if $d=4$, then $P_{3,-1}=\big((0,0,0,0),(0,0,-1,0),(0,0,-1,-1),(1,0,-1,-1),(1,1,-1,-1)\big)$.
    It is easy checked that the $2d$ paths $P_{i,\eta}$ are vertex disjoint except for $\vec 0$, hence edge disjoint. 
    Let $w_{i,\eta}\in \{\pm 1\}^d$ be the endpoint of $P_{i,\eta}$, with $j$-th coordinate $w_{i,\eta}(j)$. 
    Let $\cO_{i,\eta}\subseteq\bbZ^d$ be the induced subgraph on the shifted orthant 
    \[
    V(\cO_{i,\eta})
    =
    \{v\in\bbZ^d~:~ w_{i,\eta}(j) v_j\geq 1~\forall 1\leq j\leq d\}.
    \]
    It is well known that each orthant $\cO_{i,\eta}$ is a transient graph: for example with appropriate bounded edge weights, random walk on an orthant is just the coordinate-wise absolute value of simple random walk on $\bbZ^d$ (or see \cite[Section 5]{chiarini2016note}). 
    Therefore the edge-disjoint graphs $G_{i,\eta}=P_{i,\eta}\cup \cO_{i,\eta}$ are transient.
    Defining $G_{\Lambda,i,\eta}$ as the wiring of $\Lambda\cap G_{i,\eta}$ with boundary $\Lambda$, the $2d$ subgraphs $G_{\Lambda,i,\eta}\subseteq \bbZ^d_{\Lambda}$ are also edge-disjoint, and each $R_{G_{\Lambda,i,\eta}}(\vec 0\lrarrow z_{\Lambda})$ is uniformly bounded (independently of $\Lambda$).

    Given $\xi\in \bbR_+^{E(\bbZ^d)}$, let $\xi_{i,\eta}$ be its restriction to $G_{\Lambda,i,\eta}$, with $\xi_{i,\eta}(e)=\infty$ for $e\notin E(G_{\Lambda,i,\eta})$. 
    Then using the same notation as in \eqref{eq:convex-tail-finish-moments}, Rayleigh's monotonicity principle shows 
    \begin{equation}
    \label{eq:Q-min-Zd}
    Q_{\xi}(\vec 0)
    \leq 
    \min_{i,\eta} Q_{\xi_{i,\eta}}(\vec 0).
    \end{equation}
    We next use \eqref{eq:domination-as-before} as before, and apply the proof of Theorem~\ref{thm:moment-bounds} to each $(i,\eta)$.
    For $\xi\sim\rho^{\otimes E(\bbZ^d)}$, we thus deduce from \eqref{eq:Q-min-Zd} that $Q_{\xi}(\vec 0)$ is bounded above by the minimum of $2d$ \iid sub-$\alpha/2$ random variables, hence is sub-$d\alpha$. 
    The last part of Proposition~\ref{prop:power-tail-convex} again completes the proof.
\end{proof}

\begin{proof}[Proof of Corollary~\ref{cor:maximum-value-Zd}]
By transitivity, the values of $R_{\bbZ^d_{\Lambda}}(v\lrarrow z_{\Lambda})\leq R_{\bbZ^d}(v\lrarrow\infty)<\infty$ are uniformly bounded.
The upper tail tightness then follows by union bounding over $v\in\Lambda$ using Theorem~\ref{thm:moment-bounds-Zd}.

The lower bound on the support follows by \eqref{eq:fluctuation-LB}, which we now justify more carefully.
Let us fix $M$ to be chosen later and condition on the event $I_M=\{\sup_{w\in N(v)}|\phi(w)|\leq M\}$; recall $N(v)$ denotes the neighborhood of $v$.
Since $0\leq U_e(x)\leq \alpha\log(x+1)+C$ is monotone, the conditional probability that $|\phi(v)|\geq M$ is at least
\[
    \frac{\bbP[\phi(v)\in [M,3M]~|~I_M]}
    {\bbP[\phi(v)\in [-M,M]~|~I_M]}
    \geq
    \frac{
    \int_{M}^{3M}
    \exp\Big(
    -2d\big(C+\alpha\log\big(1+M+\phi(v)\big)\big)
    \Big)
    ~\de\phi(v)
    }
    {\int_{-M}^{M} 1~\de\phi(v)}
    \geq 
    e^{-2dC}(1+4M)^{-2d\alpha}
    .
\]
Here $M+\phi(v)\leq 4M$ upper-bounds the absolute difference between $\phi(v)$ and $\phi(w)\in [-M,M]$.
For any $\delta>0$, the right-hand side equals  $\frac{4d}{\delta |\Lambda|}$ if 
\[
M
=
\frac{e^{-C/\alpha}(\delta|\Lambda|/4d)^{1/2d\alpha} -1}{4}
.
\]
In particular consider an independent set $\cI\subseteq\Lambda$ of size $|\cI|\geq |\Lambda|/4d$.
Conditional on the restriction $\phi|_{\cI^c}$ to the complement of $\cI$ having maximum absolute value at most $M$, conditional independence implies
\[
    \bbP
    \big[\sup_{v\in\Lambda}|\phi(v)|\leq M~|~\phi|_{\cI^c}\big]
    \leq 
    \lt(1-\frac{4d}{\delta|\Lambda|}\rt)^{|\cI|}
    \leq 
    e^{-\Omega(1/\delta)}.
\]
Casework on $\max_{w\in \cI^c}|\phi(w)|$ implies $\bbP[\sup_{v\in\Lambda}|\phi(v)|\leq M]\leq 
e^{-\Omega(1/\delta)}$.
Finally note that $M\geq 
c(\delta,d)|\Lambda|^{1/2d\alpha}$ for large enough $|\Lambda|$ depending on $\delta$.
Since $\delta$ was arbitrary, we conclude tightness for the lower tail as desired.
\end{proof}

\begin{proof}[Proof of Corollary~\ref{cor:maximum-value}]
    By transitivity, the values of $R_{G_{\Lambda}}(v\lrarrow z_{\Lambda})\leq R_G(v\lrarrow\infty)<\infty$ are uniformly bounded.
    The upper bound on the support is then immediate from Theorem~\ref{thm:stretched-exponential-bounds} by union bounding over $v\in\Lambda$.
    
    The lower bound on the support follows by the analog of \eqref{eq:fluctuation-LB}. The only property of $G$ we use is that it has finite maximum degree $D<\infty$.
    Assume $|\Lambda|\geq 2$ for convenience of writing logarithms.
    Again fix $M$ to be chosen later and condition on the event $I_M=\{\sup_{w\in N(v)}|\phi(w)|\leq M\}$.
    Since $0\leq U_e(x)\leq Cx^{\beta}$ is monotone, the conditional probability that $|\phi(v)|\geq M$ is at least 
    \[
    \frac{\bbP[\phi(v)\in [M,3M]~|~I_M]}
    {\bbP[\phi(v)\in [-M,M]~|~I_M]}
    \geq
    \frac{
    \int_{M}^{3M}
    \exp\Big(-CD(M+\phi(v))^{\beta}\Big)
    ~\de\phi(v)
    }
    {\int_{-M}^{M} 1~\de\phi(v)}
    \geq 
    \exp\Big(-CD(4M)^{\beta}\Big).
    \]
    For any $\delta>0$, this quantity is at least $\frac{2D}{\delta|\Lambda|}$ if 
    $
    M=\frac{1}{4}\lt(\frac{\log(\delta|\Lambda|/2D)}{CD}\rt)^{1/\beta}
    $.
    In particular consider an independent set $\cI\subseteq\Lambda$ of size $|\cI|\geq |\Lambda|/2D$.
    Conditional on the restriction $\phi|_{\cI^c}$ to the complement of $\cI$ having maximum absolute value at most $M$, conditional independence gives
    \[
    \bbP
    \big[\sup_{v\in\Lambda}|\phi(v)|\leq M~|~\phi|_{\cI^c}\big]
    \leq 
    \lt(1-\frac{2D}{\delta|\Lambda|}\rt)^{|\cI|}
    \leq 
    e^{-\Omega(1/\delta)}.
    \]
    Casework on $\max_{w\in \cI^c}|\phi(w)|$ implies $\bbP[\sup_{v\in\Lambda}|\phi(v)|\leq M]\leq 
    e^{-\Omega(1/\delta)}$, finishing the lower tail as before.
\end{proof}

\subsection{Proofs for Generalized Zero Boundary Conditions and Iterated Laplacians}

\begin{proof}[Proof of Theorem~\ref{thm:free-boundary}]
    We describe the changes from Theorem~\ref{thm:main}. 
    Now instead of $G_{\Lambda}$ we work on $\Lambda$ itself, viewed as an induced subgraph.
    Let $\wh\mu_{\Lambda_i,v_0,V}$ be the free boundary Gibbs measure with $\phi(v_0)=0$.
    Similarly to \eqref{eq:Z2-comparison},
    \[
    \wh\mu_{\Lambda_i,S_i,v_0,U}
    \preceq_{\con}
    \wh\mu_{\Lambda_i,S_i,v_0,V}
    \preceq_{\con}
    \wh\mu_{\Lambda_i,v_0,V}
    \preceq_{\con}
    \int 
    \gamma_{\Lambda_i,\xi}
    ~\de \rho^{\otimes E(\Lambda_i)}.
    \]
    Let $\wh\xi\sim \rho^{\otimes E(G)}$ be as before and let $\xi^{(i)}$ be its restriction to $E(\Lambda_i)$.
    Let $H_{\wh\xi}\subseteq G$ be as before with transient infinite cluster $\cC$, now unique by assumption.
    Let $w_0,w$ be closest points in $\cC$ to $v_0$ and $v$ respectively. 
    Then $R_{G,\wh\xi}(w_0\lrarrow v_0)$ and $R_{G,\wh\xi}(w\lrarrow v)$ are almost surely finite, with the same law by transitivity of $G$.
    Similarly for $R_{H_{\wh\xi}}(w\lrarrow \infty)$ and 
    $R_{H_{\wh\xi}}(\infty\lrarrow w_0)$.
    Therefore $R_{G,\wh\xi}(w\lrarrow w_0)$ is tight since 
    \[
    R_{G,\wh\xi}(w\lrarrow w_0)
    \leq 
    C_{\rho,p}^2 R_{H_{\wh\xi}}(w\lrarrow w_0)
    \leq 
    C_{\rho,p}^2
    R_{H_{\wh\xi}}(w\lrarrow \infty)
    +
    C_{\rho,p}^2 R_{H_{\wh\xi}}(\infty\lrarrow w_0)
    .
    \]
    Finally since $\Lambda_i$ exhausts $V(G)$, the convergence $\lim_{i\to\infty} R_{H_{\wh\xi}\cap\Lambda_i}(w\lrarrow w_0)=R_{H_{\wh\xi}}(w\lrarrow w_0)$ holds almost surely, hence in distribution. The same reasoning applies for $w_0\lrarrow v_0$ and $w\lrarrow v$. Both conclusions now follow easily.
\end{proof}

\begin{proof}[Proof of Theorem~\ref{thm:free-boundary-tail}]
    For parts \ref{it:alpha-eps-free} and \ref{it:beta-eps-free}, the proof is essentially the same as Theorems~\ref{thm:moment-bounds} and \ref{thm:stretched-exponential-bounds}.
    Namely for $\Lambda$ large enough depending on $(v,v_0)$ one has
    $
    R_{\Lambda}(v\lrarrow v_0)
    \leq 
    2R_{G}(v\lrarrow v_0)
    <
    \infty$
    .
    Here the latter inequality follows from transience. Furthermore if $G$ is transitive then 
    \[
    R_{G}(v\lrarrow v_0)
    \leq 
    R_{G}(v_0\lrarrow\infty)
    +
    R_{G}(v\lrarrow\infty)
    =
    2R_{G}(v_0\lrarrow \infty)
    \]
    does not depend on $v$.
    The same argument using Theorem~\ref{thm:convex-comparison} proves both \ref{it:alpha-eps-free} and \ref{it:beta-eps-free}.
\end{proof}

\begin{proof}[Proof of Theorem~\ref{thm:higher-laplacian}]
    Parts~\ref{it:membrane-moment}, \ref{it:membrane-stretch} follow exactly the same way as Theorems~\ref{thm:moment-bounds}, \ref{thm:stretched-exponential-bounds}.
    $F$ still takes the form \eqref{eq:F-abstract-form}, where the terms $M_e$ or $M_v$ now incorporate the iterated Laplacian. Then $F(1)$ again gives the Gaussian model of the same problem, where localization is known from \cite{sakagawa2003entropic}. 
    The only other place the graph structure was used in Theorems~\ref{thm:moment-bounds}, \ref{thm:stretched-exponential-bounds} was the bound $R_{G_{\Lambda}}(v\lrarrow z_{\Lambda})\leq R_G(v\lrarrow \infty)$.
    This is really just the statement that in the Gaussian free field, conditional variances only decrease when some fields are pinned to zero.
    As a general fact about Gaussian processes, this remains valid for the Gaussian fields associated with $\Delta^j$.
    Otherwise, the proofs are unchanged.

    Likewise part~\ref{it:membrane-maximum} requires almost no changes either. The upper bound again follows by union bounding over $\Lambda$.
    For the lower bound, the same proof as in Corollary~\ref{cor:maximum-value} works, except the independent set $\cI$ should have minimum pairwise distance at least $j+1$ so that $\{\phi(v)~:~v\in \cI\}$ are conditionally independent given $\phi|_{\cI^c}$.
    This is just an independent set in the distance-at-most-$j$ graph $G_j\supseteq G$, which has maximum degree at most $D^j$.
\end{proof}

\subsection*{Acknowledgment}

We thank Ron Peled for introducing us to the problem of localization and for several enlightening conversations.
Thanks also to Volker Betz, Eric Thoma and Wei Wu for helpful discussions.

\footnotesize
\bibliographystyle{plain}
\bibliography{bib}

\end{document}